\documentclass[a4paper]{amsart}
\usepackage{amsmath,amssymb,amsthm}
\usepackage{graphicx} 
\usepackage{color}
\usepackage{comment}

\renewcommand{\div}{{\rm div}\,}

\def\d{\partial}

\def\be{\mbox{\boldmath $e$}}
\def\br{\mbox{\boldmath $r$}}

\def\etab{\mbox{\boldmath $\eta$}}

\def\bN{\mbox{\boldmath $N$}}

\def\bR{\mbox{\boldmath $R$}}

\def\1{\mbox{\boldmath $1$}}
\def\0{\mbox{\boldmath $0$}}
\def\cH{\mathcal{H}}
\def\cK{\mathcal{K}}
\def\<{\langle}
\def\>{\rangle}

\theoremstyle{plain}
\newtheorem{thm}{Theorem}[section]

\newtheorem{lem}[thm]{Lemma}

\newtheorem{rmk}[thm]{Remark}
\newtheorem{ex}[thm]{Example}

\title[symmetries in thin shell theory]{Symmetries in the linear theory of the deformation of thin shells}
\author{Yoshiki Jikumaru}
\address{Faculty of Information Networking for Innovation And Design, Toyo University, 1-7-11 Akabanedai, Kita-ku, Tokyo, 115-8650, Japan}
\email{jikumaru@toyo.jp}

\begin{document}
\maketitle
\begin{abstract}
In this paper, we show that symmetries, which are known in the theory of integrable systems, naturally appeared in the classical linear theory of deformations of thin shells.
Our result shows that if the middle surface of a shell becomes `integrable', infinitely many deformations exist that have no shear strains and twisting of the coordinate lines.
\end{abstract}

\section{Introduction}
In the field of architectural surface design, shell structures have the advantage that, with an appropriate geometry, they can be thin and lightweight yet still cover large spaces without supports.
In the last two decades, remarkable connections have been discovered between the shell membrane theory and the theory of integrable systems \cite{RogersSchief,RSSLame,RSSEnneper,SKR}.
In this paper, we focus on deformations in the linear theory of thin shells under the hypothesis of Kirchhoff-Love and reveal another relationship with the integrable systems.
The strains caused by the deformation must satisfy the St. Venant type compatibility conditions in the theory of elasticity, which is known as Gol'denweizer's compatibility conditions in the thin shell theory \cite{Goldenweiser,Novozhilov}.
First, the Gol'denweizer's compatibility conditions are reformulated as an analogue of the Gauss-Mainardi-Codazzi equations for the middle surface of a shell.
It is then shown that Gol'denweizer's conditions can be regarded as ``symmetries'' of the Gauss-Mainardi-Codazzi equations under the assumption that no shear strain and twisting of the coordinate lines occurred in the deformation.
Finally, some examples of symmetries are given for the case where the Gauss equations become soliton equations, which are important in the field of integrable geometry.

\section{Preliminaries}

In the linear theory of thin shells, a shell can be analysed using classical differential geometry of surfaces by considering its middle surface as a smooth surface in 3-dimensional Euclidean space and taking parallel surfaces for the thickness.
For the basic theory of thin shells, see, for example, \cite{GreenZerna,Love,Novozhilov}.
This paper follows the symbols in \cite{Novozhilov,RogersSchief}.

Let us consider a surface $\br = \br (\alpha, \beta)$ in the 3-dimensional Euclidean space parametrized by the curvature line coordinates $(\alpha, \beta)$,
that is, the first and second fundamental forms are diagonalized.
We denote the first fundamental form $I$ by
\begin{equation}
I = A_1^2 \, d\alpha^2 + A_2^2 \, d\beta^2,\quad
A_1^2 = \br_\alpha \cdot \br_\alpha, \quad A_2^2 = \br_\beta \cdot \br_\beta,
\end{equation}
where the subscripts represent the partial derivatives.
Let $\be_1$ and $\be_2$ be the unit tangent vectors on the surface defined by the following relations:
\begin{equation}
\label{eq:tangent1}
\br_\alpha = A_1 \be_1, \quad \br_\beta = A_2 \be_2.
\end{equation}
In the curvature line coordinates, a direct calculation shows that the compatibility condition $(\br_\alpha)_\beta = (\br_\beta)_\alpha$ for the linear system \eqref{eq:tangent1} is given by
\begin{equation}
(A_1)_\beta = p A_2, \quad (A_2)_\alpha = q A_1, \quad 
(\be_1)_\beta = q \be_2, \quad (\be_2)_\alpha = p \be_1,
\end{equation}
where the first and second equations define $p$ and $q$.
Since $\bN$ is the unit vector, we have the relations
\begin{equation}
\bN_\alpha = H_\circ \be_1, \quad \bN_\beta = K_\circ \be_2,
\end{equation}
for some scalar values $H_\circ$ and $K_\circ$.
Then, the \textit{principal curvatures} $\kappa_1, \kappa_2$ and the \textit{principal curvature radii} $R_1, R_2$ are defined as follows:
\begin{equation}
H_\circ = - \kappa_1 A_1 = \frac{A_1}{R_1}, \quad K_\circ = - \kappa_2 A_2 = \frac{A_2}{R_2}.
\end{equation}
In particular, we have
\begin{equation}
\label{eq:Weingarten}
\bN_\alpha = - \kappa_1 \bR_\alpha, \quad
\bN_\beta = - \kappa_2 \bR_\beta.
\end{equation}
This formula is called the \textit{Weingarten formula} or \textit{Rodrigues formula}.
With these notations, the \textit{Gauss formulas} are given by
\begin{equation}
(\be_1)_\alpha = - p \be_2 - H_\circ \bN, \quad (\be_2)_\alpha = p \be_1, \quad 
(\be_1)_\beta = q \be_2, \quad (\be_2)_\beta = - q \be_1 - K_\circ \bN,
\end{equation}
Or equivalently, for the orthogonal frame field $\Phi = (\be_1, \be_2, \bN)$, we have the \textit{Gauss-Weingarten formula} in the matrix form:
\begin{equation}
\Phi_\alpha = \Phi L, \quad
\Phi_\beta = \Phi M,
\end{equation}
where
\begin{equation}
L = 
\begin{pmatrix}
0 & p & H_\circ \\
- p & 0 & 0 \\
- H_\circ & 0 & 0
\end{pmatrix}, \quad
M = 
\begin{pmatrix}
0 & - q & 0 \\
q & 0 & K_\circ \\
0 & - K_\circ & 0    
\end{pmatrix}.
\end{equation}
The compatibility condition for the Gauss-Weingarten system is called the \textit{Gauss-Mainardi-Codazzi equations}, which can be represented in the matrix form as follows:
\begin{equation}
L_\beta - M_\alpha = [L, M], \quad [L, M] = LM -ML,
\end{equation}
or equivalently,
\begin{equation}
\label{eq:GMC}
p_\beta + q_\alpha + H_\circ K_\circ = 0, \quad
(H_\circ)_\beta = p K_\circ, \quad
(K_\circ)_\alpha = q H_\circ.
\end{equation}
The first equation is called the Gauss equation, and the second and third equations are called the Mainardi-Codazzi equations.

\section{Symmetries in the theory of integrable systems}
\label{sec:symmetries}

This section briefly reviews the notion of symmetries in the theory of integrable systems; see, for example, \cite{Fokas1980,Matsukidaira,FokasSantini} for details.
In general, we consider an evolution equation of the type
\begin{equation}
\label{eq:non-linear}
u_t = K(u),
\end{equation}
where $K(u)$ is some functional of $u$.
Let us consider a perturbation $u \to u + \varepsilon S$ and impose the requirement that the equation \eqref{eq:non-linear} to be invariant up to order $\varepsilon$:
\begin{equation}
\label{eq:linearized}
S_t = K'(u)[S],
\end{equation}
where $K'(u)[S]$ means the Fr\'echet derivative of $K$ at the point $u$ in the direction of $S$, that is,
\begin{equation}
K'(u)[S] := \left. \frac{\d}{\d \varepsilon} K(u + \varepsilon S) \right|_{\varepsilon=0}.
\end{equation}
The equation \eqref{eq:linearized} and its solution $S=S_u$ are called the \textit{linearized equation} and a \textit{symmetry} of the original equation \eqref{eq:non-linear}, respectively.
The equation \eqref{eq:non-linear} is called \textit{exactly solvable} if it possesses infinitely many symmetries \cite{Fokas1980,Matsukidaira}.
It is known that each classical `solitonic' equation, such as the Korteweg-de Vries equation, has an infinite number of symmetries, and symmetries are closely related to the conserved quantities of the equation.
We will show that symmetries naturally appear in the strains of the shell.
Therefore, if the middle surface of a shell is `solitonic', there are infinitely many deformations with no shear strain and twisting of the coordinate lines in the deformation.

\section{The strain-displacement relations}

We consider the surface $\br = \br(\alpha, \beta)$ as the middle surface of a thin shell model, and the shell itself can be modelled by the parallel surfaces $\br + z \bN$ ($-\delta/2 < z < \delta/2$), where $\delta$ is the thickness of the shell.
One can consider the deformation of the middle surface as follows:
\begin{equation}
\bR = \br +  \Delta, \quad \Delta = u \be_1 + v \be_2 + w \bN,
\end{equation}
where $\Delta$ is called the \textit{displacement vector}.
Under the hypothesis of Kirchhoff-Love, the displacements completely determine the deformation of the shell.
For the deformed surface $\bR$, we have the following relations:
\begin{equation}
\label{eq:deform_tangent}
\frac{1}{A_1} \bR_\alpha = (1 + \varepsilon_1) \be_1 + \omega_1 \be_2 - \vartheta \bN, \quad
\frac{1}{A_2} \bR_\beta = \omega_2 \be_1 + (1 + \varepsilon_2) \be_2 - \psi \bN,
\end{equation}
where
\begin{align}
\begin{aligned}
\varepsilon_1 &= \frac{1}{A_1} (u_\alpha + p v + H_\circ w), \quad
\varepsilon_2 = \frac{1}{A_2} (v_\beta + qu + K_\circ w), \\
\omega_1 &= \frac{1}{A_1} (v_\alpha - pu), \quad
\omega_2 = \frac{1}{A_2} (u_\beta - q v), \\
\vartheta &= \frac{1}{A_1} (- w_\alpha + H_\circ u), \quad
\psi = \frac{1}{A_2} (- w_\beta + K_\circ v).
\end{aligned}
\end{align}
The above relations are called the \textit{strain-displacement relations}.
Here, $\varepsilon_1$ and $\varepsilon_2$ are called the \textit{normal strains}, and $\omega = \omega_1 + \omega_2$ is called the \textit{shear strain}.
$\vartheta$ and $\psi$ correspond to the deflection angle under the deformation.

A direct calculation shows the following lemma (see \cite{Novozhilov}, p.28):
\begin{lem}
The compatibility condition $(\bR_\alpha)_\beta = (\bR_\beta)_\alpha$ holds if and only if the following equations are satisfied:
\begin{equation}
\label{eq:compati_R12}
\begin{aligned}
(A_1 \varepsilon_1)_\beta - (A_2 \omega)_\alpha - \varepsilon_2 (A_1)_\beta
    &= - ((\omega_1)_\alpha + H_\circ \psi) A_2, \\
(A_2 \varepsilon_2)_\alpha - (A_1 \omega)_\beta - \varepsilon_1 (A_2)_\alpha
    &= - ((\omega_2)_\beta + K_\circ \vartheta) A_1. 
\end{aligned}
\end{equation}
\end{lem}
\begin{rmk}
Note that the third component of the condition $(\bR_\alpha)_\beta = (\bR_\beta)_\alpha$ identically holds:
\begin{equation}
(A_1 \vartheta)_\beta - (A_2 \psi)_\alpha + \left( \frac{\omega_1}{R_2} - \frac{\omega_2}{R_1} \right) A_1 A_2 = 0.
\end{equation}
Moreover, we have the relation
\begin{equation}
\label{eq:compati_R3}
\frac{\psi_\alpha-p\psi}{A_1} + \frac{\omega_2}{R_1}
= \frac{\vartheta_\beta-q\psi}{A_2} + \frac{\omega_1}{R_2},
\end{equation}
which is used to define the twisting of the coordinate lines, see \eqref{eq:k1k2tau}.
\end{rmk}
Using the equations \eqref{eq:compati_R12}, we define the quantities $P$ and $Q$ as follows:
\begin{equation}
\label{eq:deformed_PQ}
\begin{aligned}
P &:= - ((\omega_1)_\alpha + H_\circ \psi)
= \frac{1}{A_2} ((A_1 \varepsilon_1)_\beta - (A_2 \omega)_\alpha - \varepsilon_2 (A_1)_\beta),\\
Q &:= - ((\omega_2)_\beta + K_\circ \vartheta)
= \frac{1}{A_1} ((A_2 \varepsilon_2)_\alpha - (A_1 \omega)_\beta - \varepsilon_1 (A_2)_\alpha).
\end{aligned}
\end{equation}

\section{The linear theory of thin shells and the hypothesis of Kirchhoff-Love}

In the following, we ignore the second-order terms of displacements and their derivatives; that is, we consider the linear theory of thin shells.
In this case, the line element $A_1' \, d\alpha$ and $A_2' \, d\beta$ along the coordinate lines on the deformed middle surface $\bR = \bR(\alpha, \beta)$ can be represented as follows:
\begin{equation}
A_1' = A_1 (1 + \varepsilon_1), \quad
A_2' = A_2 (1 + \varepsilon_2).
\end{equation}
Moreover, if we define
\begin{equation}
\be_1' = \frac{\bR_\alpha}{A_1'}, \quad
\be_2' = \frac{\bR_\beta}{A_2'}, \quad
\bN' = \be_1' \times \be_2',
\end{equation}
then the frame field $\Phi' = (\be_1', \be_2', \bN')$ on the deformed surface $\bR$ becomes
\begin{equation}
\Phi' = \Phi (E + \Omega), \quad
\Omega =
\begin{pmatrix}
0 & \omega_2 & \vartheta\\
\omega_1 & 0 & \psi\\
- \vartheta & - \psi & 0
\end{pmatrix},
\end{equation}
where $E$ denotes the identity matrix.
Note that if $\omega = \omega_1 + \omega_2 = 0$, then the matrix $\Omega$ becomes skew-symmetric.
Under the hypothesis of Kirchhoff-Love, the displacements $u_{(z)}$, $v_{(z)}$, $w_{(z)}$ for the layer $\br^{(z)} = \br + z \bN$ are determined by
\begin{equation}
u_{(z)} = u + z \vartheta, \quad
v_{(z)} = v + z \psi, \quad
w_{(z)} = w.
\end{equation}
Then the normal strains $\varepsilon_1^{(z)}$, $\varepsilon_2^{(z)}$ and the shear strain $\omega^{(z)}$ on the layer $\br^{(z)}$ are given by
\begin{equation}
\varepsilon_1^{(z)} = \frac{\varepsilon_1 - zk_1}{1+z/R_1}, \quad
\varepsilon_2^{(z)} = \frac{\varepsilon_2 - zk_2}{1+z/R_2}, \quad
\omega^{(z)} = \frac{(1-z^2 \cK) \omega + 2 (1 + z \cH) z \tau}{1 - 2 \cH z + \cK z^2},
\end{equation}
see \cite{Novozhilov} for details.
Here, the quantities $k_1$, $k_2$ and $\tau$ are defined by
\begin{equation}
\label{eq:k1k2tau}
\begin{aligned}
k_1 &= - \frac{\vartheta_\alpha + p \psi}{A_1}, \quad
k_2 = - \frac{\psi_\beta + q \vartheta}{A_2}, \\
\tau &= \frac{\psi_\alpha - p \vartheta}{A_1} + \frac{\omega_2}{R_1}
= \frac{\vartheta_\beta - q \psi}{A_2} + \frac{\omega_1}{R_2},
\end{aligned}
\end{equation}
and
\begin{equation}
\cH = -\frac{1}{2} \left( \frac{1}{R_1} + \frac{1}{R_2} \right) = \frac{1}{2} (\kappa_1 + \kappa_2), \quad
\cK = \frac{1}{R_1R_2} = \kappa_1 \kappa_2,
\end{equation}
are the mean and Gauss curvatures of the middle surface $\br$, respectively.
Here, in the definition of $\tau$, we use the identity \eqref{eq:compati_R3}.
The quantities $k_1$, $k_2$ and $\tau$ characterize the change of curvature and twist of the middle surface of a shell as it deforms (see \cite{Novozhilov}, p.25).

The six strains $\varepsilon_1$, $\varepsilon_2$, $\omega$, $k_1$, $k_2$ and $\tau$ are linear functions of the displacements $u$, $v$, $w$ and their derivatives, and these quantities completely determine the deformation of the shell, since, in general, the coefficients of the first and second fundamental forms have six components.
The strains are determined by a deformation of the middle surface; however, for the deformation to be determined from the strain, certain compatibility conditions must be satisfied.
The conditions for the strains are called St. Venant's conditions in the theory of elasticity, which are known as Gol'denweizer's condition in the linear theory of thin shells \cite{Goldenweiser,Novozhilov}.

%
\section{The compatibility condition of the strains}

In this section, we reformulate the result of Gol'denweizer as an analogue of the Gauss-Mainardi-Codazzi equations.
First, a brief overview of our strategy is given.
Using the equation \eqref{eq:compati_R12}, the Gauss-Weingarten type formula for the frame $\Phi'$ on the deformed surface can be represented as follows:
\begin{equation}
\Phi_\alpha' = \Phi' (L + L'), \quad
\Phi_\beta' = \Phi' (M + M'),
\end{equation}
where the explicit form of the matrices $L'$ and $M'$ will be given in Lemma \ref{lem:2}.
The compatibility condition $(\Phi_\alpha')_\beta = (\Phi_\beta')_\alpha$ gives, as the $0$-th order term, the Gauss-Mainardi-Codazzi equations
\[
L_\beta - M_\alpha = [L, M],
\]
and Gol'denweizer's result can be achieved in the first-order term:
\[
L_\beta' - M_\alpha' = [L', M] + [L, M'].
\]
Therefore, the purpose of this section is to describe the above components explicitly.
\begin{lem}
\label{lem:1}
The Gauss-Weingarten type formula for the frame $\Phi'$ is given by
\begin{equation}
\begin{aligned}
\Phi_\alpha' &= \Phi' (L + L'), \quad L' = \Omega_\alpha + [L, \Omega], \\
\Phi_\beta' &= \Phi' (M + M'), \quad M' = \Omega_\beta + [M, \Omega],
\end{aligned}
\end{equation}
where we put $[X, Y] = XY - YX$.
\end{lem}
\begin{proof}
Note that the relation $(E+\Omega)^{-1} = E - \Omega$ in the linear theory, we have
\begin{align*}
\Phi_\alpha'
&= \Phi_\alpha (E+\Omega) + \Phi \Omega_\alpha
= \Phi (L + L \Omega + \Omega_\alpha)
= \Phi' (E-\Omega) (L + L \Omega + \Omega_\alpha)\\
&= \Phi' (L + L \Omega + \Omega_\alpha - \Omega L)
= \Phi' (L + \Omega_\alpha + [L, \Omega]),
\end{align*}
The formula for $\Phi_\beta'$ can be proved in the same way.
\end{proof}
\begin{lem}
\label{lem:2}
The explicit forms of $L'$ and $M'$ are given as follows:
\begin{equation}
\begin{aligned}
L' &= 
\begin{pmatrix}
p \omega & P + \omega_\alpha & H_\circ'\\
-P & - p \omega & \tau A_1 - H_\circ \omega\\
-H_\circ' & - \tau A_1 & 0
\end{pmatrix}, \\
M' &= 
\begin{pmatrix}
- q\omega & -Q & \tau A_2 - K_\circ \omega \\
Q + \omega_\beta & q \omega & K_\circ\\
- \tau A_2 & -K_\circ' & 0
\end{pmatrix},
\end{aligned}
\end{equation}
where we put $H_\circ' = - k_1 A_1$ and $K_\circ' = - k_2 A_2$.
\end{lem}
\begin{proof}
Note that we have the following relations from \eqref{eq:k1k2tau}:
\begin{equation}
\label{eq:k1k2tau2}
\begin{aligned}
&H_\circ' = - k_1 A_1 = \vartheta_\alpha + p \psi, \quad
K_\circ' = - k_2 A_2 = \psi_\beta + q \vartheta, \\
&\tau A_1 = \psi_\alpha - p\vartheta + H_\circ \omega_2, \quad
\tau A_2 = \vartheta_\beta - q\psi + K_\circ \omega_1,
\end{aligned}
\end{equation}
A direct calculation shows
\begin{equation}
\begin{aligned}
L' &= 
\begin{pmatrix}
p\omega & (\omega_2)_\alpha -H_\circ \psi & \vartheta_\alpha + p\psi\\
(\omega_1)_\alpha + H_\circ \psi & - p \omega & \psi_\alpha - p \vartheta - H_\circ \omega_1\\
- (\vartheta_\alpha +p\psi) & - \psi_\alpha + p \vartheta - H_\circ \omega_2 & 0
\end{pmatrix},\\
M' &=
\begin{pmatrix}
- q\omega & (\omega_2)_\beta + K_\circ \vartheta & \vartheta_\beta -q \psi - K_\circ \omega_2\\
(\omega_1)_\beta - K_\circ \vartheta & q\omega & \psi_\beta + q\vartheta\\
- \vartheta_\beta + q\psi -K_\circ \omega_1 & - (\psi_\beta + q \vartheta) & 0
\end{pmatrix}.
\end{aligned}
\end{equation}
Using the relations \eqref{eq:deformed_PQ} and \eqref{eq:k1k2tau2} gives the desired result.
\end{proof}
\begin{thm}[Gol'denweizer's compatibility condition]
The relation $L_\beta' - M_\alpha' = [L', M] + [L, M']$ is equivalent to the following system of equations:
\begin{equation}
\label{eq:Goldenweizer}
\begin{aligned}
&P_\beta + Q_\alpha + \omega_{\alpha\beta} + H_\circ' K_\circ + H_\circ K_\circ' = 0,\\
&(H_\circ')_\beta - (PK_\circ + p K_\circ') = (\tau A_2)_\alpha + \tau (A_2)_\alpha - \omega (K_\circ)_\alpha, \\
&(K_\circ')_\alpha - (QH_\circ + q H_\circ') = (\tau A_1)_\beta + \tau (A_1)_\beta - \omega (H_\circ)_\beta.
\end{aligned}
\end{equation}
\end{thm}
\begin{proof}
Using the Mainardi-Codazzi equations \eqref{eq:GMC}${}_{2,3}$, we have
\begin{equation}
L_\beta' = 
\begin{pmatrix}
p_\beta \omega + p \omega_\beta & P_\beta + \omega_{\alpha\beta} & (H_\circ')_\beta \\
-P_\beta & - p_\beta \omega - p \omega_\beta & (\tau A_1)_\beta - p\omega K_\circ - H_\circ \omega_\beta\\
- (H_\circ')_\beta & - (\tau A_1)_\beta & 0
\end{pmatrix},
\end{equation}
\begin{equation}
M_\alpha' = 
\begin{pmatrix}
- q_\alpha \omega - q \omega_\alpha & - Q_\alpha & (\tau A_2)_\alpha - q \omega H_\circ - K_\circ \omega_\alpha\\
Q_\alpha + \omega_{\alpha\beta} & q_\alpha \omega + q \omega_\alpha & (K_\circ')_\alpha\\
- (\tau A_2)_\alpha & - (K_\circ')_\alpha & 0
\end{pmatrix},
\end{equation}
Moreover, applying the Gauss equation \eqref{eq:GMC}${}_1$, we have
\begin{equation}
\label{eq:Goldenweizer1}
\begin{aligned}
&\quad L_\beta' - M_\alpha'\\
&=
\begin{pmatrix}
p \omega_\beta + q\omega_\alpha - H_\circ K_\circ \omega & P_\beta + Q_\alpha + \omega_{\alpha\beta} & (H_\circ')_\beta + \omega_\alpha K_\circ - (\tau A_2)_\alpha + q \omega H_\circ\\
-P_\beta-Q_\alpha-\omega_{\alpha\beta} & H_\circ K_\circ \omega - p \omega_\beta - q \omega_\alpha & (\tau A_1)_\beta - p \omega K_\circ - H_\circ \omega_\beta - (K_\circ')_\alpha\\
(\tau A_2)_\alpha - (H_\circ')_\beta & (K_\circ')_\alpha - (\tau A_1)_\beta & 0
\end{pmatrix}.
\end{aligned}
\end{equation}
Since a direct calculation shows
\[
\begin{aligned}
[L', M] &= 
\begin{pmatrix}
q \omega_\alpha & - H_\circ' K_\circ -2 pq \omega & K_\circ (P+\omega_\alpha) - q(H_\circ \omega - \tau A_1)\\
H_\circ' K_\circ - 2pq \omega & H_\circ K_\circ \omega - q \omega_\alpha & - p \omega K_\circ - q H_\circ'\\
- q \tau A_1 - P K_\circ & q H_\circ' - p \omega K_\circ & - H_\circ K_\circ \omega
\end{pmatrix},\\
[L, M'] &= 
\begin{pmatrix}
p \omega_\beta - H_\circ K_\circ \omega & 2 pq \omega - H_\circ K_\circ' & p K_\circ' + q \omega H_\circ\\
2pq \omega + H_\circ K_\circ' & - p\omega_\beta & p (K_\circ \omega - \tau A_2) - H_\circ (Q + \omega_\beta)\\
q \omega H_\circ - pK_\circ' & Q H_\circ + p \tau A_2 & H_\circ K_\circ \omega
\end{pmatrix},
\end{aligned}
\]
we have
\begin{equation}
\label{eq:Goldenweizer2}
[L', M] + [L, M'] =
\begin{pmatrix}
p\omega_\beta + q \omega_\alpha - H_\circ K_\circ \omega & - H_\circ'K_\circ - H_\circ K_\circ' & k + \omega_\alpha K_\circ\\
H_\circ'K_\circ + H_\circ K_\circ' & H_\circ K_\circ \omega - p\omega_\beta-q\omega_\alpha & -h-\omega_\beta H_\circ\\
q\omega H_\circ - k & h - p\omega K_\circ & 0
\end{pmatrix},
\end{equation}
where we use the following notations to save the space:
\begin{equation}
h = QH_\circ + qH_\circ' + p\tau A_2, \quad
k = PK_\circ + pK_\circ' + q\tau A_1,
\end{equation}
Comparing the components in \eqref{eq:Goldenweizer1} and \eqref{eq:Goldenweizer2} gives the following three equations:
\begin{equation}
\begin{aligned}
&P_\beta + Q_\alpha + \omega_{\alpha\beta} + H_\circ' K_\circ + H_\circ K_\circ' = 0,\\
&(H_\circ')_\beta - (\tau A_2)_\alpha = PK_\circ + p K_\circ' + q\tau A_1 - q \omega H_\circ, \\
&(K_\circ')_\alpha - (\tau A_1)_\beta = QH_\circ + q H_\circ' + p\tau A_2 - p \omega K_\circ.
\end{aligned}
\end{equation}
The definition of $p$, $q$ and the Mainardi-Codazzi equations \eqref{eq:GMC}${}_{2,3}$ give the desired result.
\end{proof}

\section{The symmetry properties in the case $\omega = \tau = 0$.}

From the Lemma \ref{lem:2}, if we assume $\omega = \tau = 0$, that is, no shear strain and the twisting of the coordinate lines occurred under the deformation, the Gauss-Weigarten type formula for the frame $\Phi' = (\be_1', \be_2', \bN')$ is given by
\begin{equation}
\begin{aligned}
\Phi_\alpha' &= \Phi' (L + L'), \quad
L = 
\begin{pmatrix}
0 & p  & H_\circ\\
-p & 0 & 0\\
- H_\circ & 0 & 0
\end{pmatrix}, \quad
L' = 
\begin{pmatrix}
0 & P  & H_\circ'\\
-P & 0 & 0\\
- H_\circ' & 0 & 0
\end{pmatrix}, \\
\Phi_\beta' &= \Phi' (M + M'), \quad
M = 
\begin{pmatrix}
0 & -q & 0\\
q & 0 & K_\circ\\
0 & -K_\circ & 0
\end{pmatrix}, \quad
M' = 
\begin{pmatrix}
0 & -Q & 0\\
Q & 0 & K_\circ'\\
0 & -K_\circ' & 0
\end{pmatrix}.
\end{aligned}
\end{equation}
Recall that the Gauss-Mainardi-Codazzi equations are given by the compatibility condition $(\Phi_\alpha)_\beta = (\Phi_\beta)_\alpha$, that is, the matrix equation
\begin{equation}
L_\beta - M_\alpha = [L, M],
\end{equation}
and the Gol'denweizer's compatibility conditions are given by
\begin{equation}
L_\beta' - M_\alpha' = [L', M] + [L, M'].
\end{equation}
Here, if we consider the prime notation ${}^\prime$ as the `Fr\'echet derivative', 
we find that the `Fr\'echet derivative' of the Gauss-Mainardi-Codazzi system is nothing but Gol'denweizer's system.
More explicitly, for the Gauss-Mainardi-Codazzi equations
\begin{equation}
\begin{aligned}
p_\beta + q_\alpha + H_\circ K_\circ &= 0, \\
(H_\circ)_\beta &= p K_\circ, \\
(K_\circ)_\alpha &= q H_\circ,
\end{aligned}
\end{equation}
we have the following `linearized' equations (see \S \ref{sec:symmetries}):
\begin{equation}
\begin{aligned}
(S_p)_\beta + (S_q)_\alpha + S_{H_\circ} K_\circ + H_\circ S_{K_\circ} &= 0, \\
(S_{H_\circ})_\beta &= S_p K_\circ + p S_{K_\circ}, \\
(S_{K_\circ})_\alpha &= S_q H_\circ, + q S_{H_\circ},
\end{aligned}
\end{equation}
where we denote a symmetry of the function $u$ as $S_u$.
Comparing these equations with Gol'denweizer's equations \eqref{eq:Goldenweizer} for $\omega = \tau = 0$, they have symmetries 
\begin{equation}
(S_p, S_q, S_{H_\circ}, S_{K_\circ}) = (P, Q, H_\circ', K_\circ').
\end{equation}
By using the first fundamental forms and curvatures of the middle surface of the shell $\br$, the above symmetries are equivalent to
\begin{equation}
S_{A_1} = A_1 \varepsilon_1, \quad
S_{A_2} = A_2 \varepsilon_2, \quad
S_{\kappa_1} = k_1 - \varepsilon_1 \kappa_1, \quad
S_{\kappa_2} = k_2 - \varepsilon_2 \kappa_2.
\end{equation}
In conclusion, we have the following theorem:
\begin{thm}
For the coefficients of the first fundamental forms $A_1$, $A_2$ and the principal curvatures $\kappa_1$, $\kappa_2$, if there exist their symmetries $S_{A_1}$, $S_{A_2}$, $S_{\kappa_1}$, $S_{\kappa_2}$, then the strains $\varepsilon_1$, $\varepsilon_2$, $k_1$ and $k_2$ defined by
\begin{equation}
\label{eq:strain_symmetry}
\varepsilon_1 = \frac{S_{A_1}}{A_1}, \quad
\varepsilon_2 = \frac{S_{A_2}}{A_2}, \quad
k_1 = S_{\kappa_1} + \frac{S_{A_1}}{A_1} \kappa_1, \quad
k_2 = S_{\kappa_2} + \frac{S_{A_2}}{A_2} \kappa_2,
\end{equation}
satisfy the Gol'denweizer's compatibility condition for $\omega = \tau = 0$.
\end{thm}
Finally, we give some examples of the theorem for classes of surfaces whose Gauss equations are `solitonic', sometimes referred to as `integrable' surfaces.
\begin{ex}[Minimal surfaces]
If the middle surface of a shell $\br$ is minimal, that is, the mean curvature $\cH$ vanishes everywhere, we have
\begin{equation}
A_1 = A_2 = e^u, \quad
H_\circ = e^{-u}, \quad
K_\circ = - e^{-u},
\end{equation}
in the conformal curvature line coordinates \cite{RSSLame}.
From the Gauss equation \eqref{eq:GMC}${}_1$, the function $u$ must satisfy the Liouville equation
\begin{equation}
u_{\alpha\alpha} + u_{\beta\beta} = e^{-2u}.
\end{equation}
By taking the Fr\'echet derivative, we have the linearized equation
\begin{equation}
S_{\alpha\alpha} + S_{\beta\beta} + 2 e^{-2u} S = 0,
\end{equation}
which is the Schr\"odinger type equation.
A solution $S$ gives a symmetry of the Liouville equation, and for any such symmetry, one can check
\begin{equation}
S_{A_1} = S_{A_2} = S e^u, \quad
S_{\kappa_1} = 2 S e^{-2u}, \quad
S_{\kappa_2} = - 2 S e^{-2u}.
\end{equation}
Hence, from the equations \eqref{eq:strain_symmetry}, the strains
\begin{equation}
\varepsilon_1 = \varepsilon_2 = S, \quad
k_1 = S e^{-2u}, \quad
k_2 = - S e^{-2u},
\end{equation}
satisfy the Gol'denweizer's compatibility condition with $\omega = \tau = 0$.
\end{ex}
\begin{ex}[Constant mean curvature (CMC) surfaces]
If the middle surface of a shell has non-zero constant mean curvature $\cH$, we have
\begin{equation}
A_1 = A_2 = e^u, \quad
H_\circ = -2 \cH \sinh u, \quad
K_\circ = -2 \cH \cosh u,
\end{equation}
in the conformal curvature line coordinates.
From the Gauss equation \eqref{eq:GMC}${}_1$, the function $u$ must satisfy the elliptic sinh-Gordon equation
\begin{equation}
u_{\alpha\alpha} + u_{\beta\beta} + 4\cH^2 \sinh u \cosh u = 0.
\end{equation}
The linearized equation becomes the Schr\"odinger type equation:
\begin{equation}
S_{\alpha\alpha} + S_{\beta\beta} + 4 \cH^2 \cosh (2u) S = 0.
\end{equation}
Hence, in the same manner as before, for a solution $u$ of the elliptic sinh-Gordon equation and any symmetry $S$, the strains defined by
\begin{equation}
\varepsilon_1 = \varepsilon_2 = S, \quad
k_1 = 2 \cH S e^{-u} \cosh u, \quad
k_2 = 2 \cH S e^{-u} \sinh u,
\end{equation}
satisfy the Gol'denweizer's compatibility condition with $\omega = \tau = 0$.
\end{ex}
\begin{ex}[Pseudo-spherical surfaces]
If the middle surface of a shell has constant negative Gaussian curvature $\cK = -1/\rho^2$, which is sometimes referred to as a pseudo-spherical surface, we have
\begin{equation}
A_1 = \cos u, \quad
A_2 = \sin u, \quad
H_\circ = - \frac{1}{\rho} \sin u, \quad
K_\circ = \frac{1}{\rho} \cos u,
\end{equation}
in the curvature line coordinates \cite{Rogers_Schief_2002}.
From the Gauss equation, the function $u$ must satisfy the (hyperbolic) sine-Gordon equation
\begin{equation}
u_{xx} - u_{yy} = \frac{1}{\rho^2} \sin u \cos u.
\end{equation}
A symmetry $S$ is a solution of the linearized equation
\begin{equation}
S_{xx} - S_{yy} = \frac{1}{\rho^2} \cos (2u) S.
\end{equation}
Therefore, for any symmetry $S$ of the sine-Gordon equation, the strains defined by
\begin{equation}
\varepsilon_1 = - S \tan u, \quad
\varepsilon_2 = S \cot u, \quad
k_1 = k_2 = \frac{S}{\rho},
\end{equation}
satisfy the Gol'denweizer's compatibility condition with $\omega = \tau = 0$.
\end{ex}

\section*{Appendix}

One can check that the result \eqref{eq:Goldenweizer} and the equations (5.6) that appeared in the book \cite{Novozhilov} are certainly equivalent with some appropriate replacements.
Note that the quantities $P$, $Q$ defined in \eqref{eq:deformed_PQ} can be rewritten as follows:
\begin{equation}
\begin{aligned}
&P= \frac{1}{A_2} \left( A_1 (\varepsilon_1)_\beta + (A_1)_\beta (\varepsilon_1 - \varepsilon_2) - \frac{1}{2} A_2 \omega_\alpha - (A_2)_\alpha \omega \right) - \frac{1}{2} \omega_\alpha,\\
&Q= \frac{1}{A_1} \left( A_2 (\varepsilon_2)_\alpha + (A_2)_\alpha (\varepsilon_2 - \varepsilon_1) - \frac{1}{2} A_1 \omega_\beta - (A_1)_\beta \omega \right) - \frac{1}{2} \omega_\beta.
\end{aligned}
\end{equation}
Therefore the equation \eqref{eq:Goldenweizer}${}_1$ becomes
\[
\begin{aligned}
\frac{-k_1}{R_1} + \frac{-k_2}{R_2} &+ \frac{1}{A_1A_2} 
\left[ \frac{\d}{\d \alpha} \left\{ \frac{1}{A_1} \left( A_2 (\varepsilon_2)_\alpha + (A_2)_\alpha (\varepsilon_2 - \varepsilon_1) - \frac{1}{2} A_1 \omega_\beta - (A_1)_\beta \omega \right)  \right\} \right.\\
&+ \left. \frac{\d}{\d \beta} \left\{ \frac{1}{A_2} \left( A_1 (\varepsilon_1)_\beta + (A_1)_\beta (\varepsilon_1 - \varepsilon_2) - \frac{1}{2} A_2 \omega_\alpha - (A_2)_\alpha \omega \right) \right\} \right] = 0,
\end{aligned}
\]
which is equivalent to the equation (5.6${}^{\prime\prime\prime}$) in \cite{Novozhilov} with the replacement $(-k_1, -k_2) \mapsto (\kappa_1, \kappa_2)$ and $(\alpha, \beta) \mapsto (\alpha_1, \alpha_2)$.
In the same way, the equations \eqref{eq:Goldenweizer}${}_{2,3}$ become
\[
\begin{aligned}
(A_1 (-k_1))_\beta - (-k_2) (A_1)_\beta &- (A_2\tau)_\alpha - \tau (A_2)_\alpha + \frac{\omega}{R_1} (A_2)_\alpha\\
&- \frac{1}{R_2} \left( (A_1 \varepsilon_1)_\beta - (A_2 \omega)_\alpha - \varepsilon_2 (A_1)_\beta \right) = 0,\\
(A_2 (-k_2))_\alpha - (-k_1) (A_2)_\alpha &- (A_1\tau)_\beta - \tau (A_1)_\beta + \frac{\omega}{R_2} (A_1)_\beta\\
&- \frac{1}{R_1} \left( (A_2 \varepsilon_2)_\alpha - (A_1 \omega)_\beta - \varepsilon_1 (A_2)_\alpha \right) = 0,
\end{aligned}
\]
which are equivalent to the equations (5.6${}^{\prime}$), (5.6${}^{\prime\prime}$) in \cite{Novozhilov}, respectively.

\section*{Acknowledgments}
This study is supported by JST CREST Grant Number JPMJCR1911 and JSPS KAKENHI Grant Number JP24K16924.

\bibliographystyle{amsplain}
\bibliography{main}
\end{document}